\documentclass[11pt]{amsart}
\usepackage{enumerate,mathtools,amssymb, enumitem}
\usepackage{hyperref}

\theoremstyle{plain}
\newtheorem{theorem}{Theorem}

\newtheorem{proposition}[theorem]{Proposition}
\theoremstyle{definition}
\newtheorem{question}[theorem]{Question}

\newtheorem{remark}[theorem]{Remark}

\newcommand{\C}{\ensuremath{\mathbb{C}}}

\newcommand{\id}{\mathrm{id}}

\begin{document}

\title{Decomposing nuclear maps}

\author[J. Carri\'{o}n]{Jos\'{e} R. Carri\'{o}n}
\address{Department of Mathematics, Texas
  Christian University, Fort Worth, Texas 76129, USA.}
\email{j.carrion@tcu.edu}

\author[C. Schafhauser]{Christopher Schafhauser}
\address{Department of Mathematics,
University of Nebraska--Lincoln,
Lincoln, Nebraska 68588, USA}
\email{cschafhauser2@unl.edu}

\date{\today}

\begin{abstract}
  We show that the strengthened version of the completely positive
  approximation property of Brown, Carri\'{o}n, and White---where the
  downward maps are asymptotically order zero and the upward maps are
  convex combinations of order zero maps---is enjoyed by every nuclear
  order zero map.
\end{abstract}

\maketitle

Choi and Effros \cite{Choi-Effros78} and Kirchberg \cite{Kirchberg77}
proved that nuclearity for $C^*$-algebras is equivalent to the
completely positive approximation property: a $C^*$-algebra $A$ is
nuclear if and only if there exist nets of finite dimensional $C^*$-algebras
$F_i$ and of contractive completely positive (c.c.p.) maps
\begin{equation}
  A \xrightarrow{\psi_i} F_i \xrightarrow{\phi_i} A
\end{equation}
the composition of which tends to the identity map on $A$ in the
point-norm topology, in the sense that
\begin{equation}
  \label{eq:1}
  \| (\phi_i\psi_i)(a) - a\| \to 0 \quad\text{for all } a\in A.
\end{equation}
See \cite{Brown-Ozawa08} for the general theory of nuclearity for
$C^*$-algebras.

This approximation property has been strengthened by several authors
in various cases \cite{Blackadar-Kirchberg97, Winter-Zacharias10,
  Hirshberg-Kirchberg-etal12, Brown-Carrion-etal16}.  Notably, in
\cite{Brown-Carrion-etal16}, building on techniques from
\cite{Hirshberg-Kirchberg-etal12}, it was shown that all nuclear
$C^*$-algebras enjoy completely positive approximations as in
\eqref{eq:1}, where in addition all of the maps $\phi_i$ are convex
combinations of order zero maps and the maps $\psi_i$ are
asymptotically order zero, in the sense that $\|\psi_i(a)\psi_i(b)\|
\to 0$ whenever $a$ and $b$ are positive elements of $A$ with $ab=0$.

In this short note, we show that this improved approximation property
holds for any c.c.p. nuclear order zero map between two
$C^*$-algebras.  Recall that a c.c.p. map $\theta \colon A \rightarrow
B$ has \emph{order zero} if $\theta(a) \theta(b) = 0$ whenever $a$ and
$b$ are positive elements in $A$ with $ab = 0$; see
\cite{Winter-Zacharias09b}.

\begin{theorem}
  \label{thm:main}
  Let $A$ and $B$ be $C^*$-algebras.  If $\theta \colon A \rightarrow B$
  is a c.c.p. nuclear order zero map, then there are nets of finite
  dimensional $C^*$-algebras $(F_i)$ and of c.c.p. maps
  $A \xrightarrow{\psi_i} F_i \xrightarrow{\phi_i} B$ satisfying
  \begin{enumerate}
    \item $\| (\phi_i\psi_i)(a) - \theta(a) \| \to 0$
      for all $a\in A$,
    \item $\|\psi_i(a)\psi_i(b)\| \to 0$ for all $a,b\in A$ such that
      $ab=0$, and
    \item every $\phi_i$ is a convex combination of c.c.p. order zero
      maps.
  \end{enumerate}
\end{theorem}

The refined versions of the completely positive approximation property
originate in a search for a theory of non-commutative covering
dimension of $C^*$-algebras in the work of Kirchberg and Winter on
decomposition rank \cite{Kirchberg-Winter04} and of Winter and
Zacharias on nuclear dimension \cite{Winter-Zacharias09b}.  In
both \cite{Kirchberg-Winter04} and \cite{Winter-Zacharias09b} one
finds conditions asking for an approximate factorization of the
identity map similar to that in Theorem~\ref{thm:main}, but with more
control on the upward maps $\phi_i$.

The refined approximation property given in
\cite{Brown-Carrion-etal16}---which is Theorem~\ref{thm:main} in the
special case when $\theta$ is the identity map on a nuclear
$C^*$-algebra---is a crucial part of \cite{CETWW}, where it is shown
that unital, separable, simple, nuclear, $\mathcal{Z}$-stable
$C^*$-algebras have nuclear dimension at most 1.  This latter result
plays an important role in the classification theory for nuclear
$C^*$-algebras through the results of \cite{Winter16a}. The improved
approximation properties mentioned above have also seen application to
near-inclusions of $C^*$-algebras \cite{Hirshberg-Kirchberg-etal12}.
Given the growing interest in the structure and classification of
$^*$-homomorphisms (as opposed to $C^*$-algebras) in papers such as
\cite{Dadarlat04, Gabe18, Schafhauser18, CETWW, CGSTW1,
  BosaGabeSimsWhite19}, we hope that the results presented here will
prove just as useful as their counterparts for $C^*$-algebras have.
\bigskip

The outline of the proof of Theorem~\ref{thm:main} is not altogether
new: the same broad steps used in \cite{Hirshberg-Kirchberg-etal12}
and \cite{Brown-Carrion-etal16} apply.  We point out, however, that
the proofs here, even once restricted to the case of the identity map
on a nuclear $C^*$-algebra, are new.  For instance, we avoid the use
of amenable and quasidiagonal traces, which are critical in
\cite{Brown-Carrion-etal16}.

In order to prove Theorem~\ref{thm:main}, we will prove the following
von Neumann algebraic version first.  Theorem~\ref{thm:main} will then follow
from a Hahn-Banach argument similar to that used in \cite{Brown-Carrion-etal16}
and \cite{Hirshberg-Kirchberg-etal12}.

\begin{theorem}
  \label{thm:weak-approximations}
  Let $A$ be a $C^*$-algebra and let $N$ be a von Neumann algebra.  If
  $\theta \colon A \to N$ is a weakly nuclear $^*$-homomorphism, then
  there are nets of finite dimensional $C^*$-algebras $(F_i)$ and of
  c.c.p. maps $A \xrightarrow{\psi_i} F_i \xrightarrow{\phi_i} N$
  satisfying
  \begin{enumerate}
  \item $(\phi_i\psi_i)(a) \to \theta(a)$ in the $\sigma$-strong$^*$
    topology (and therefore also in the $\sigma$-weak topology) for
    all $a\in A$, \label{item:wapps-1}
  \item $\| \psi_i(a) \psi_i(b) \| \to 0$ for all $a, b\in A$ such
    that $ab = 0$, and \label{item:wapps-3}
    \item every $\phi_i$ is a $^*$-homomorphism.\label{item:wapps-2}

  \end{enumerate}
\end{theorem}

In fact, the approximate factorizations satisfying conditions (i) and
(ii) of Theorem~\ref{thm:weak-approximations} can be arranged for all
c.c.p. weakly nuclear maps.

\begin{proposition}
  \label{prop:downward-maps}
  Let $A$ be a $C^*$-algebra and let $N$ be a von Neumann
  algebra.  If $\theta \colon A \to N$ is a c.c.p. weakly nuclear map,
  then there are nets of finite dimensional $C^*$-algebras $(F_i)$
  and of c.c.p. maps $A \xrightarrow{\psi_i} F_i \xrightarrow{\phi_i} N$
  satisfying
  \begin{enumerate}
    \item $(\phi_i\psi_i)(a) \to \theta(a)$ in the $\sigma$-strong$^*$
      topology for all
    $a\in A$, and
    \item $\| \psi_i(a) \psi_i(b) \| \to 0$ for all $a, b\in A$ such
      that $ab = 0$.
  \end{enumerate}
  If $A$ is quasidiagonal, we may further arrange for
  \begin{enumerate}
  \item[\normalfont (ii$'$)] $\|\psi_i(a) \psi_i(b) - \psi_i(ab)\|
    \rightarrow 0$ for all $a, b \in A$.
  \end{enumerate}
\end{proposition}

\begin{remark}
  \label{rem:sigma-strong-top}
  If $N$ is a von Neumann algebra and $\rho$ is a normal state on $N$,
  then the seminorm $\|\cdot\|_{2,\rho}^\sharp$ on $N$ is defined by
  \begin{equation}
    \|x\|_{2,\rho}^\sharp
    = \rho\Big( \frac{xx^*+x^*x}{2} \Big)^{1/2},
    \quad x\in N.
  \end{equation}
  If $\{ \rho_i \}$ is a separating family of normal states on $N$,
  then the topology generated by $\{ \|\cdot\|^\sharp_{2, \rho_i} \}$
  agrees with the $\sigma$-strong$^*$ topology on any bounded subset
  of $N$; see \cite[III.2.2.19]{Blackadar06}.
\end{remark}

\begin{proof}[Proof of Proposition~\ref{prop:downward-maps}]
  First we assume $A$ is quasidiagonal and prove the stronger
  approximation condition.  We may assume $A$ is unital.  Let
  $\mathcal{F} \subseteq A \setminus \{0\}$ and $\mathcal{S}
  \subseteq S(N) \cap N_*$ be finite sets, and let $\varepsilon > 0$.
  Because $\theta$ is weakly nuclear, there exist an integer
  $\ell \geq 1$ and u.c.p. maps $A \xrightarrow{\psi'} M_\ell(\mathbb{C})
  \xrightarrow{\phi'} N$ satisfying
  \begin{equation}
    \label{eq:13}
    \| \phi'(\psi'(x)) - \theta(x) \|_{2, \rho}^\sharp <
    \|x\| \varepsilon
  \end{equation}
  for all $x \in \mathcal{F}$ and $\rho \in \mathcal{S}$.
  We will prove the existence of an integer $k$ and of u.c.p.
  maps $A \xrightarrow{\psi} M_k(\mathbb{C}) \xrightarrow{\phi} N$
  satisfying
  \begin{enumerate}[label = (\alph*)]
    \item $\| \phi(\psi(x)) - \theta(x) \|_{2, \rho}^\sharp < 5 \|x\| \varepsilon$
      for all $x \in \mathcal{F}$ and $\rho \in \mathcal{S}$, and
    \item $\| \psi(x) \psi(y) - \psi(xy) \| < \|x\| \|y\| \varepsilon$
      for all $x, y \in \mathcal{F}$.
  \end{enumerate}

  Replacing $A$ with $C^*(\mathcal{F} \cup \{1_A\}) \subseteq A$, we may assume
  $A$ is separable.  Fix a separable Hilbert space $H$ and a faithful
  representation $A \subseteq \mathcal{B}(H)$ with $A \cap
  \mathcal{K}(H) = 0$.  After identifying $M_\ell(\C)$ with $\mathcal{B}(\mathbb{C}^\ell)$,
  Voiculescu's Theorem\footnote{See \cite[Theorem~II.5.3]{Davidson96}
    for the version we are using here.} provides an isometry $v \colon
  \mathbb{C}^\ell \to H$ such that $\| v^* x v - \psi'(x) \| < \|x\|
  \varepsilon$ for all $x \in \mathcal{F}$.

  Since $A$ is quasidiagonal,
  there is a finite rank projection $p \in \mathcal{B}(H)$ such that
  $\| p x - x p \| < \varepsilon \|x\|$ for all $x \in \mathcal{F}$ and
  $\| p v - v \| < \varepsilon$.  Identify $\mathcal{B}(pH)$ with
  $M_k(\C)$ for some integer $k \geq 1$ and define $\psi \colon A \to
  M_k(\C)$ by $\psi(a) = p a p$.  Then
  \begin{equation}
    \label{eq:2}
    \| v^* \psi(x) v - \psi'(x) \| < 3 \|x\| \varepsilon
  \end{equation}
  and
  \begin{equation}
    \label{eq:4}
    \| \psi(x) \psi(y) - \psi(x y) \| < \|x\| \|y\| \varepsilon
  \end{equation}
  for all $x, y \in \mathcal{F}$.  Note that \eqref{eq:13} and \eqref{eq:2} imply
  \begin{equation}
    \label{eq:3}
    \| \phi'(v^*\psi(x)v) - \theta(x) \|_{2,\rho}^\sharp < 5\| x \| \varepsilon
  \end{equation}
  for all $x\in \mathcal{F}$ and $\rho \in \mathcal{S}$.  To complete
  the proof (in the quasidiagonal case), define $\phi =
  \phi'(v^*(\cdot)v)$.

  Now we handle the case of a general $C^*$-algebra $A$.  Let $\pi
  \colon C_0(0, 1] \otimes A \to A$ be the $^*$-homomorphism
  given by $\pi(f \otimes a) = f(1) a$ for all $f \in C_0(0, 1]$, $a
  \in A$.  By the homotopy invariance of quasidiagonality \cite{Voiculescu91},
  $C_0(0, 1] \otimes A$ is quasidiagonal.
  Applying the first part of the proof to the c.c.p. weakly nuclear map
  $\theta \pi$, there is a net $C_0(0, 1] \otimes A \xrightarrow{\psi_i'}
  F_i \xrightarrow{\phi_i} N$ satisfying (i) and (ii$'$).  The result
  follows by defining $\psi_i \colon A \to F_i$ by $\psi_i(a) =
  \psi_i'(\id_{(0,1]} \otimes a)$.
\end{proof}

\begin{remark}
  \label{rem:downward-maps}
  If $A$ and $B$ are $C^*$-algebras and $\theta \colon A \rightarrow
  B$ is a c.c.p. nuclear map, then a statement analogous to
  Proposition~\ref{prop:downward-maps} holds with the approximations
  in (i) taken in the operator norm.  The proof is the same except
  with operator norm estimates in \eqref{eq:13}, \eqref{eq:3}, and
  (a).
\end{remark}

\begin{proof}[Proof of Theorem~\ref{thm:weak-approximations}]
  Through a standard argument, we may assume $A$ is separable and $N$
  has separable predual.  Note that $N$ decomposes as a direct sum $N
  = N_1 \oplus N_\infty$ with $N_1$ finite and $N_\infty$ properly
  infinite.  Then the map $\theta$ decomposes accordingly as $\theta =
  \theta_1 \oplus \theta_\infty$ for weakly nuclear $^*$-homomorphisms
  $\theta_i \colon A \rightarrow N_i$, $i \in \{1, \infty\}$.  By
  handling each summand separately, we may assume that either $N$ is
  finite or $N$ is properly infinite.

  \textbf{Properly Infinite Case:} We may assume $A$ is unital.  Fix a
  faithful normal state $\rho$ on $N$, a finite set $\mathcal{F}
  \subseteq A$ of unitaries, and a finite set $\mathcal{G} \subseteq A
  \times A$ such that $xy = 0$ for all $(x, y) \in \mathcal{G}$, and
  let $\varepsilon > 0$.  By Proposition~\ref{prop:downward-maps},
  there are a finite dimensional $C^*$-algebra $F$ and c.c.p. maps $A
  \xrightarrow{\psi} F \xrightarrow{\phi'} N$ satisfying
  \begin{enumerate}[label = (\alph*)]
  \item $\| \phi'(\psi(x)) - \theta(x) \|_{2,\rho}^\sharp <
    \varepsilon$ for all $x \in \mathcal{F}$, and
  \item $\| \psi(x) \psi(y) \| < \varepsilon$ for all $(x, y) \in
    \mathcal{G}$.
  \end{enumerate}

  At this point we can follow the proof of Lemma~2.4 of
  \cite{Brown-Carrion-etal16}, which borrows heavily from Theorem~2.2
  of \cite{Haagerup85}.  Indeed the last two paragraphs of the proof
  of Lemma~2.4 of \cite{Brown-Carrion-etal16}, with $\theta$ and
  $\phi'$ in place of $\pi_\infty$ and $\theta$, produce a
  $^*$-homomorphism $\phi \colon F \rightarrow M$ such that
  \begin{equation}
    \label{eq:5}
    \| (\phi\psi)(x) - \theta(x)\|_{2,\rho}^\sharp < 2 \varepsilon^{1/2}
  \end{equation}
  for all $x\in \mathcal{F}$.

  \textbf{Finite Case:}
  Let $\tau_N$ be a faithful normal trace on $N$ so that on
  bounded subsets of $N$, the $\sigma$-strong$^*$ topology is induced
  by the norm
  \begin{equation}
    \label{eq:9}
    \|x\|_{2, \tau_N} = \tau_N(x^*x)^{1/2},
    \quad x \in N.
  \end{equation}
  Let $\tau_A = \tau_N \theta$, and note that $\theta$ induces a
  faithful normal $^*$-homomorphism $\bar{\theta} \colon
  \pi_{\tau_A}(A)'' \to N$.  As $N$ is finite, there is normal
  expectation $N \to \bar{\theta}(\pi_{\tau_A}(A)'')$ (see
  Lemma~1.5.11 in \cite{Brown-Ozawa08}).  Therefore, the corestriction
  of $\pi_{\tau_A}$ to $\pi_{\tau_A}(A)''$ is weakly nuclear.  It
  follows that $\pi_{\tau_A}(A)''$ is hyperfinite by the equivalence
  of (5) and (6) in Theorem~3.2.2 of \cite{Brown06}.  After replacing
  $N$ with $\pi_{\tau_A}(A)''$ and $\theta$ with $\pi_{\tau_A}$, we
  may assume that $N$ is hyperfinite, which we do for the rest of the
  proof.

  Let $F_i$ be an increasing sequence of finite dimensional
  $C^*$-subalgebras of $N$ with $\sigma$-strong$^*$ dense union and
  let $\phi_i \colon F_i \to N$ denote the inclusion maps.  Let
  $\tau_i = \tau_{N}|_{F_i}$ be the induced trace on $F_i$.  Fix a
  trace-preserving expectation $E_i \colon N \to F_i$ and set $\psi_i'
  = E_i \theta$. The sequence $(\psi_i')$ induces a $^*$-homomorphism
  to the tracial von Neumann algebra ultraproduct
  \begin{equation}
    \label{eq:10}
    \psi^\omega \colon A \to F^\omega = \prod^\omega (F_i,
    \tau_i)
  \end{equation}
  and a c.c.p. map to the $C^*$-algebra ultraproduct
  \begin{equation}
    \label{eq:11}
    \psi_\omega \colon A \to F_\omega = \prod_\omega
    F_i.
  \end{equation}
  Let $q \colon F_\omega \to F^\omega$ be the quotient map and note
  that $q \psi_\omega = \psi^\omega$.  Let $J = \ker q$.

  Proposition~4.6 of \cite{Kirchberg-Rordam14} proves that the uniform
  2-norm null sequences in a norm ultrapower form a $\sigma$-ideal
  (Definition 4.4 of \cite{Kirchberg-Rordam14}), by a standard
  application of Kirchberg's $\epsilon$-test.  A very slight
  modification of the proof shows that $J$ is a $\sigma$-ideal of
  $F_\omega$.  Then there is a positive contraction $e \in J \cap
  \psi_\omega(A)'$ such that $e x = x$ for all $x \in J \cap
  C^*(\psi_\omega(A))$.

  Consider the c.c.p. map $\Psi\colon A\to F_\omega$ defined by
  $\Psi(x) = (1-e) \psi_\omega (x)$, $x\in A$.  We claim $\Psi$ is
  order zero.  Indeed, if $x, y\in A_+$ satisfy $xy=0$, then
  $q(\psi_\omega(x)\psi_\omega(y)) = \psi^\omega(xy) = 0$, so
  $\psi_\omega(x)\psi_\omega(y) \in J$.  Then $\Psi(x)\Psi(y) =
  (1-e)^2 \psi_\omega(x)\psi_\omega(y) = (1-e)\big(
  \psi_\omega(x)\psi_\omega(y) - \psi_\omega(x)\psi_\omega(y)\big) =
  0$, as required.

  Represent $e$ by a sequence of positive
  contractions $(e_i)$ with $e_i \in F_i$ and define
  \begin{equation}
    \label{eq:12}
    \psi_i = (1 - e_i)^{1/2} \psi'_i(\cdot) (1 - e_i)^{1/2} \colon A
    \to F_i.
  \end{equation}
  Then the maps $A \xrightarrow{\psi_i} F_i \xrightarrow{\phi_i} N$
  have the desired properties.
\end{proof}

Finally, we prove Theorem~\ref{thm:main}.

\begin{proof}[Proof of Theorem~\ref{thm:main}]
  The structure theorem for order zero maps (see Corollary~4.1 of
  \cite{Winter-Zacharias09b}) implies there is a $^*$-homomorphism
  $\acute{\theta} \colon C_0(0, 1] \otimes A \rightarrow B$ such that
  $\theta(a) = \acute{\theta}(\mathrm{id}_{(0,1]} \otimes a)$ for $a
  \in A$.  By Theorem~2.10 in \cite{Gabe17}, $\acute{\theta}$ is
  nuclear (since $\theta$ is nuclear).  Hence, after replacing $A$
  with $C_0(0, 1] \otimes A$ and $\theta$ with $\acute{\theta}$, we
  may assume $\theta$ is a $^*$-homomorphism.

  Let $\mathcal{F}$ be a finite subset of $A$ and let $\varepsilon >
  0$.  Let $K$ be the subset of $\mathcal{B}(A,B)$ consisting of all
  c.c.p. maps $\eta\colon A\to B$ for which there exist a finite
  dimensional $C^*$-algebra $F$ and c.c.p. maps $A\xrightarrow{\psi} F
  \xrightarrow{\phi} B$ satisfying:
  \begin{itemize}
  \item $\eta = \phi\psi$,
  \item $\|\psi(a)\psi(b)\| < \varepsilon$ for all
    $a, b\in \mathcal{F}$ with $a b = 0$, and
  \item $\phi$ is a convex combination of c.c.p. order zero maps.
  \end{itemize}
  To prove the theorem, it suffices to show that $\theta$ is in the
  point-norm closure of $K$.

  By Theorem~\ref{thm:weak-approximations}, there are nets of finite
  dimensional $C^*$-algebras $(F_i)$ and c.c.p. maps $A
  \xrightarrow{\psi_i} F_i \xrightarrow{\phi_i} B^{**}$ satisfying
  conditions (i)--(iii) of Theorem~\ref{thm:weak-approximations}.
  Lemma~1.1 of \cite{Hirshberg-Kirchberg-etal12} implies that for each
  $i$ there is a net $(\phi_i^\lambda\colon F_i\to B)_{\lambda\in
    \Lambda}$ of c.c.p. order zero maps such that $\phi_i^\lambda(x)
  \to \phi_i(x)$ in the $\sigma$-strong$^*$ topology (and therefore in
  the $\sigma$-weak topology) for all $x\in F_i$.  Therefore, we may
  assume that the image of $\phi_i$ is contained in $B$ for all $i$.
  Then, for all sufficiently large $i$, we have that $\phi_i \circ
  \psi_i \in K$.  Thus $\theta$ is in the point-weak closure of $K$.
  As $K$ is a convex set in $\mathcal{B}(A,B)$, its point-norm and
  point-weak closures coincide, by the Hahn-Banach Theorem\footnote{%
    In fact, this is consequence of the Hahn-Banach Separation
    Theorem.  See Lemma~2.3.4 of \cite{Brown-Ozawa08} for a proof.  }%
  .  Next, recall that the weak topology on $B$ is the same as the
  restriction to $B$ of the weak-$^*$ topology on $B^{**}$ (with $B$
  regarded as a subspace of $B^{**}$ via the standard map).  But the
  weak$^*$-topology on $B^{**}$ is precisely the $\sigma$-weak
  topology (see e.g. \cite[I.8.6]{Blackadar06}), so the theorem
  follows.
\end{proof}

For a nuclear $^*$-homomorphism $\theta \colon A \rightarrow B$ as in
Theorem~\ref{thm:main}, it is natural to consider a stronger
approximation property where the maps $\psi_i$ are required to be
approximately multiplicative in the sense that
\begin{enumerate}
\item[(ii$'$)] $\| \psi_i(a) \psi_i(b) - \psi_i(ab) \| \rightarrow 0$
  for all $a, b \in A$.
\end{enumerate}

\begin{question}\label{question}
  For which nuclear $^*$-homomorphisms $\theta \colon A \rightarrow B$
  between $C^*$-algebras $A$ and $B$ are there approximate
  factorizations $A \xrightarrow{\psi_i} F_i \xrightarrow{\phi_i} B$
  of $\theta$ as in Theorem~\ref{thm:main} that also satisfy
  {\normalfont(ii$'$)}?
\end{question}

A characterization of such $\theta$ is likely related to some
appropriate form of quasidiagonality.  In particular, note that if
$\theta$ is faithful and $(\phi_i)$, $(\psi_i)$, and $(F_i)$ are as in
Question~\ref{question}, then for all $a \in A$, we have
\begin{equation}\label{eq:qd}
  \|\psi_i(a)\| \geq \|\phi_i(\psi_i(a))\| \rightarrow \|\theta(a)\| = \|a\|.
\end{equation}
Now, the net $(\psi_i)$ is an approximate embedding of $A$ into finite
dimensional $C^*$-algebras, and this shows that $A$ is
quasidiagonal. We end with some further comments and partial results
related to the Question~\ref{question}.

First, Theorem~2.2 of \cite{Brown-Carrion-etal16} states that, for a
separable nuclear $C^*$-algebra $A$, $\id_A$ enjoys approximate
factorizations as in Question~\ref{question} if and only if $A$ is
quasidiagonal and all traces on $A$ are quasidiagonal.  It is not hard
to show that the separability assumption is superfluous.

Second, if $A$ is quasidiagonal and $B$ has no traces, then every
nuclear $^*$-homomorphism $\theta \colon A \rightarrow B$ enjoys
approximate factorizations as in Question~\ref{question}.  This
follows as in the proof of Theorem~\ref{thm:main} with only slight
modifications.  The key observations are that
\begin{enumerate}[label*=(\alph*)]
\item because $A$ is quasidiagonal, the approximate factorizations of
  the composition $A \xrightarrow{\theta} B \hookrightarrow B^{**}$
  given in Proposition~\ref{prop:downward-maps} can be taken to
  satisfy (ii$'$), and
\item because $B$ has no traces, $B^{**}$ is properly infinite, and so
  the finite case in the proof of
  Theorem~\ref{thm:weak-approximations} is not needed.
\end{enumerate}

Third, one might also compare Question~\ref{question} with the recent
results of \cite{BosaGabeSimsWhite19}, which address the nuclear
dimension and decomposition rank of $\mathcal{O}_\infty$-stable and
$\mathcal{O}_2$-stable $^*$-homomorphisms.  In particular, Theorem~D
of \cite{BosaGabeSimsWhite19} states that a full
$\mathcal{O}_2$-stable $^*$-homomorphism with a separable and exact
domain has nuclear dimension zero if and only if it is nuclear and its
domain is quasidiagonal.  Noting that decomposition rank zero is
equivalent to nuclear dimension zero, Proposition~1.7 of
\cite{BosaGabeSimsWhite19} implies such a map enjoys approximate
factorizations as in Question~\ref{question}.  Again, via standard
reductions, the separability of $A$ is not needed in this result.

\textbf{Acknowledgments.}  The authors are happy to record their
gratitude to Stuart White, who asked the question leading to the main
result of this note.

\bibliographystyle{amsplain}


\providecommand{\bysame}{\leavevmode\hbox to3em{\hrulefill}\thinspace}
\providecommand{\MR}{\relax\ifhmode\unskip\space\fi MR }
\providecommand{\MRhref}[2]{%
  \href{http://www.ams.org/mathscinet-getitem?mr=#1}{#2}
}
\providecommand{\href}[2]{#2}

\end{document}